\documentclass[A4paper]{amsart}
\usepackage{amsmath,amsthm,amssymb}
\usepackage[all]{xy}
\begin{document}

\newtheorem{theorem}{Theorem}[section]
\newtheorem{lemma}[theorem]{Lemma}
\newtheorem{proposition}[theorem]{Proposition}
\newtheorem{corollary}[theorem]{Corollary}

\theoremstyle{definition}
\newtheorem{definition}[theorem]{Definition}
\newtheorem{example}[theorem]{Example}
\newtheorem{formula}[theorem]{Formula}
\newtheorem{nothing}[theorem]{}
\newtheorem{exercise}[theorem]{Exercise}

\theoremstyle{remark}
\newtheorem{remark}[theorem]{Remark}

\numberwithin{equation}{section}

\renewcommand\arraystretch{1.2}

\newcommand{\twist}{{\pmb{\mathtt{tw}}}}
\newcommand{\contr}{{\mspace{1mu}\lrcorner\mspace{1.5mu}}}
\newcommand{\per}{{\mspace{-1mu}\cdot\mspace{-1mu}}}
\newcommand{\dual}{^{\vee}}
\newcommand{\de}{\partial}
\newcommand{\debar}{{\overline{\partial}}}
\newcommand{\desude}[2]{{\dfrac{\de #1}{\de #2}}}

\newcommand{\mapor}[1]{{\stackrel{#1}{\longrightarrow}}}
\newcommand{\ormap}[1]{{\stackrel{#1}{\longleftarrow}}}

\newcommand{\mapver}[1]{\Big\downarrow\vcenter{\rlap{$\scriptstyle#1$}}}
\newcommand{\vermap}[1]{\Big\uparrow\vcenter{\rlap{$\scriptstyle#1$}}}

\newcommand{\binfty}{\boldsymbol{\infty}}
\newcommand{\bi}{\boldsymbol{i}}
\newcommand{\vale}[1]{$[#1]$}

\newcommand\tensor{{\textstyle\bigotimes}}
\newcommand\somdir{{\textstyle\bigoplus}}
\newcommand\external{{\textstyle\bigwedge}}
\newcommand\symmetric{{\textstyle\bigodot}}

\renewcommand{\bar}{\overline}
\renewcommand{\Hat}[1]{\widehat{#1}}

\newcommand{\sA}{\mathcal{A}}
\newcommand{\Oh}{\mathcal{O}}
\newcommand{\sH}{\mathcal{H}}
\newcommand{\sI}{\mathcal{I}}
\newcommand{\sE}{\mathcal{E}}
\newcommand{\sF}{\mathcal{F}}
\newcommand{\sK}{\mathcal{K}}
\newcommand{\sL}{\mathcal{L}}
\newcommand{\sM}{\mathcal{M}}
\newcommand{\sG}{\mathcal{G}}
\newcommand{\sX}{\mathcal{X}}
\newcommand{\sB}{\mathcal{B}}
\newcommand{\sY}{\mathcal{Y}}

\newcommand{\Q}{\mathbb{Q}}
\newcommand{\C}{\mathbb{C}}
\newcommand{\Z}{\mathbb{Z}}
\newcommand{\K}{\mathbb{K}}
\newcommand{\N}{\mathbb{N}}
\newcommand{\ope}[1]{\operatorname{#1}}
\newcommand{\ad}{\operatorname{ad}}
\newcommand{\MC}{\operatorname{MC}}
\newcommand{\Def}{\operatorname{Def}}
\newcommand{\hDef}{\widetilde{\operatorname{Def}}}
\newcommand{\Hom}{\operatorname{Hom}}
\newcommand{\End}{\operatorname{End}}
\newcommand{\Image}{\operatorname{Im}}
\newcommand{\Der}{\operatorname{Der}}
\newcommand{\Mor}{\operatorname{Mor}}
\newcommand{\Hilb}{\operatorname{Hilb}}
\newcommand{\Cone}{\operatorname{Cone}}
\newcommand{\DER}{{{\mathcal D}er}}
\newcommand{\Coder}{{\operatorname{Coder}}}
\newcommand{\Id}{\operatorname{Id}}

\newcommand{\coker}{\operatorname{Coker}}
\newcommand{\NA}{\mathbf{NA}}
\newcommand{\ide}[1]{\mathfrak{#1}}
\newcommand{\Aut}{\operatorname{Aut}}
\newcommand{\copa}{\mathfrak{a}}
\newcommand{\copl}{\mathfrak{l}}
\newcommand{\copc}{\mathfrak{c}}

\renewcommand{\subjclassname}{%
\textup{2010} Mathematics Subject Classification}

%%%%%

\title{A relative version of the ordinary perturbation lemma}
\author{Marco Manetti}
\date{February 3, 2010}

\maketitle

\begin{abstract}
The perturbation lemma and the homotopy transfer for $L_{\infty}$-algebras is proved in a elementary way by using a relative version of the ordinary perturbation lemma for chain complexes and the coalgebra perturbation lemma.

\end{abstract}

\section{Introduction}

Let $N$ be a differential graded vector space and 
let $M\subset N$ be a differential graded subspace such that the inclusion map
$\imath\colon M\to N$ is a quasi-isomorphism. The basic homology theory shows that there exists a homotopy $h\colon N\to N$ such that 
$Id+dh+hd\colon N\to N$ is a projection onto $M$.  
If $\tilde{d}$ is a new differential on $N$ such that $\de=\tilde{d}-d$ is ``small'' in some appropriate sense, then the \emph{ordinary perturbation lemma} (Theorem~\ref{thm.pertlemma}) gives explicit functorial formulas, in terms of $\de$ and $h$, for 
a differential $\tilde{D}$ on $M$ and for  an injective morphism of differential graded vector spaces 
$\tilde{\imath}\colon (M,\tilde{D})\to (N,\tilde{d})$.\par

Has been pointed out by Huebschmann and Kadeishvili \cite{HK} that if $M,N$ are  differential graded (co)algebra, and $h$ is a (co)algebra homotopy (Definition~\ref{def.coalgebracontraction}), then also  
$\tilde{\imath}$ is a morphism of differential graded (co)algebras. This assumption are verified for instance when we consider the tensor coalgebras generated by $M,N$ and the natural extension of $h$ to $T(N)$ (this fact is referred as \emph{tensor trick} in the literature). Therefore the ordinary perturbation lemma can be easily used to prove Kadeishvili's theorem \cite{Kad80,Kad82} on the homotopy transfer of 
$A_{\infty}$ structures (see also \cite{HK,hueainfty,KoSo,KonSoi,transfer,Merk}).\par

If we wants to use the same strategy for $L_{\infty}$-algebras, we have to face
the following problems:\begin{enumerate}

\item the tensor trick breaks down for symmetric 
powers and coalgebra homotopies are not stable under symmetrization,

\item not every $L_{\infty}$-algebra is the symmetrization of an $A_{\infty}$-algebra.

\end{enumerate}
Therefore the proof of the homotopy transfer for $L_{\infty}$-algebras requires either a nontrivial additional work \cite{HS,pertlie,pertshlie} or a different approach, see e.g. \cite{fuka,K} and the arXiv version of \cite{cone}.

The aim of this paper is to show that the homotopy transfer for $L_{\infty}$-algebras (Theorem~\ref{thm.symmcoalgebrapert}) follows easily  from a slight 
modification (Theorem~\ref{thm.relapert}) of the ordinary perturbation lemma  in which we assume that 
$\tilde{d}$ is a differential when restricted to a differential graded subspace $A\subset N$ satisfying suitable properties. 
 
The paper is written in a quite elementary style and we do not assume any knowledge of homological perturbation theory. We only assume that the reader is familiar with the basic properties of graded tensor and graded symmetric coalgebras. The bibliography contains the documents that have been more useful  in the writing of this paper and it is necessarily incomplete; for  more complete references the reader may consult 
\cite{jimmurra,hueainfty}. I apologize in advance for every possible misattribution of previous results.

\bigskip
\section{The category of contractions}

Let $R$ be a fixed commutative ring;  by a differential graded $R$-module
we mean a $\mathbb{Z}$-graded $R$-module $N=\oplus_{i\in \mathbb{Z}}N^i$ together a $R$-linear differential $d_N\colon N\to N$ of degree $+1$. 
 
Given two differential graded $R$-modules
$M,N$ we denote by $\Hom_R^n(M,N)$ the module of $R$-linear maps of degree $n$:
\[ \Hom_R^n(M,N)=\{f\in\Hom_{R}(M,N)\mid f(M_i)\subset N_{i+n},\, \forall\; i\in\Z\}.\]
Notice that $\Hom_R^0(M,N)$ are the morphisms of graded $R$-modules and
\[\{f\in \Hom_R^0(M,N)\mid d_Nf=fd_M\}\] 
is the set of cochain maps (morphisms of differential graded $R$-modules).

\begin{definition}[Eilenberg and Mac Lane {\cite[p. 81]{eilmactw}}]
A \emph{contraction}
is the data
\[ (\xymatrix{M\ar@<.4ex>[r]^\imath&N\ar@<.4ex>[l]^\pi}, h)\]
where $M,N$ are differential graded $R$-modules, $h\in\Hom^{-1}_R(N,N)$ and $\imath,\pi$ are cochain maps such that:
\begin{enumerate}

    \item (deformation retraction) $\;\pi\imath=\Id_{M}$,
    $\;\imath\pi-\operatorname{Id}_{N}=d_{N}h+hd_{N}$,

    \item (annihilation properties) $\;\pi h=h\imath=h^{2}=0$.
\end{enumerate}
\end{definition}

\begin{remark} In the original definition Eilenberg and Mac Lane do not require $h^2=0$; however, if $h$ satisfies the remaining 4 conditions, then $h'=hd_N h$ satisfies also the fifth (cf. \cite[Rem. 2.1]{pertshlie}).\end{remark}

\begin{definition}\label{def.morfismocontrazioni}
A \emph{morphism} of contractions
\[ f\colon
(\xymatrix{M\ar@<.4ex>[r]^\imath&N\ar@<.4ex>[l]^\pi}, h)\to
(\xymatrix{A\ar@<.4ex>[r]^i&B\ar@<.4ex>[l]^p}, k)\]
is a morphism of differential graded $R$-modules $f\colon N\to B$ such that $fh=kf$.
Given a morphism of contractions as above we denote by $\hat{f}\colon M\to A$ the 
morphism of differential graded $R$-modules $\hat{f}=pf\imath$.  
\end{definition}

In the notation of Definition~\ref{def.morfismocontrazioni} it is easy to see that
the diagrams

\[ \xymatrix{M\ar[r]^{\hat{f}}\ar[d]^\imath&A\ar[d]^i\\
N\ar[r]^f&B}\qquad\qquad
\xymatrix{N\ar[r]^{f}\ar[d]^\pi&B\ar[d]^p\\
M\ar[r]^{\hat{f}}&A}
\]
are commutative. In fact 
\[ i\hat{f}=ipf\imath=f\imath+(d_Bkf+kd_Bf)\imath=
f\imath+f(d_Nh+hd_N)\imath=f\imath+f(\imath\pi-\operatorname{Id}_N)\imath=
f\imath,\]
\[ \hat{f}\pi=pf\imath\pi=pf(\operatorname{Id}_N+d_Nh+hd_N)=
pf+p(d_Bk+kd_B)f=pf+p(ip-\operatorname{Id}_B)=pf.\]

\begin{definition}\label{def.composizionecontrazioni}
The \emph{composition} of contractions is defined as
\[ 
(\xymatrix{M\ar@<.4ex>[r]^\imath&N\ar@<.4ex>[l]^\pi}, h)\circ
(\xymatrix{N\ar@<.4ex>[r]^i&P\ar@<.4ex>[l]^p}, k)=
(\xymatrix{M\ar@<.4ex>[r]^{i\imath}&P\ar@<.4ex>[l]^{\pi p}}, k+ihp)
\]
\end{definition}

\bigskip

Given two contractions
$(\xymatrix{M\ar@<.4ex>[r]^\imath&N\ar@<.4ex>[l]^\pi}, h)$ and 
$(\xymatrix{A\ar@<.4ex>[r]^i&B\ar@<.4ex>[l]^p}, k)$ we  define
their tensor product as
\[ (\xymatrix{M\otimes_RA\ar@<.4ex>[r]^{\imath\otimes i}&
N\otimes_RB\ar@<.4ex>[l]^{\pi\otimes p}}, h\ast k),
\qquad h\ast k=\imath\pi\otimes k+h\otimes \operatorname{Id}_B.\]%
It is straightforward to verify that the tensor product of two contractions is a contraction, it is bifunctorial and,  up to the canonical isomorphism $(L\otimes_R M)\otimes_R N\cong L\otimes_R (M\otimes_R N)$, it is associative.\par

Given a contraction $(\xymatrix{M\ar@<.4ex>[r]^\imath&N\ar@<.4ex>[l]^\pi}, h)$, its tensor $n$th power is 

 \[ \tensor^n_R(\xymatrix{M\ar@<.4ex>[r]^\imath&N\ar@<.4ex>[l]^\pi}, h)=
 (\xymatrix{M^{\otimes n}\ar@<.4ex>[r]^{\imath^{\otimes n}}&N^{\otimes n}
 \ar@<.4ex>[l]^{\pi^{\otimes n}}}, T^nh),\]
where
\[ T^nh=\sum_{i=1}^{n}(\imath\pi)^{\otimes i-1}
\otimes h\otimes \operatorname{Id}_N^{\otimes n-i}.\]

%\begin{remark} In the above notation, it is generally false that $T^{n}h$ 
%preserves the submodule of symmetric tensors. 
%In particular the category of contraction is monoidal 
%\cite{MacLane} but not symmetric monoidal (tensor). 
%\end{remark}

The tensor product allows to define naturally the notion of algebra and coalgebra contraction; we consider here only the case of coalgebras.

\begin{definition}\label{def.coalgebracontraction}
Let $N$ be a differential graded coalgebra over a commutative ring $R$  with coproduct $\Delta\colon N\to N\otimes_R N$. 
We shall say that a contraction $(\xymatrix{M\ar@<.4ex>[r]^\imath&N\ar@<.4ex>[l]^\pi}, h)$
is a \emph{coalgebra contraction} if
\[ \Delta \colon (\xymatrix{M\ar@<.4ex>[r]^\imath&N\ar@<.4ex>[l]^\pi}, h)\to
(\xymatrix{M\otimes_RM\ar@<.4ex>[r]^{\imath\otimes \imath}&
N\otimes_RN\ar@<.4ex>[l]^{\pi\otimes \pi}}, h\ast h)\]
is a morphism of contractions. 
\end{definition}

Notice that if $\Delta$ is a morphism of contractions then $\hat{\Delta}$ is a coproduct and 
$\pi,\imath$ are morphisms of differential graded
coalgebras. Conversely,
a contraction $(\xymatrix{M\ar@<.4ex>[r]^\imath&N\ar@<.4ex>[l]^\pi}, h)$ is a coalgebra contraction if $\pi,\imath$ are morphisms of differential graded
coalgebras and
\[ (\imath\pi\otimes h+h\otimes\operatorname{Id}_N)\circ\Delta=
\Delta\circ h.\]%

\begin{example}[tensor trick]\label{exa.tensorcoalgebracontractions}
Given a contraction $(\xymatrix{M\ar@<.4ex>[r]^\imath&N\ar@<.4ex>[l]^\pi}, h)$ of differential graded $R$-modules, we can
consider the \emph{reduced tensor coalgebra} 
\[\bar{T}(N)=\somdir_{n=1}^{\infty}\tensor^n_RN\] with the 
coproduct
\[ \copa(x_1\otimes\cdots\otimes x_n)=\sum_{i=1}^{n-1}
(x_1\otimes\cdots\otimes x_i)\otimes (x_{i+1}\otimes\cdots\otimes x_n).\]
We have seen that there exists a contraction
\[(\xymatrix{\bar{T}(M)\ar@<.4ex>[r]^{T(\imath)}&\bar{T}(N)
\ar@<.4ex>[l]^{T(\pi)}}, Th)\;,\]
where $T(\imath)=\sum \imath^{\otimes n}$, $T(\pi)=\sum \pi^{\otimes n}$ and
$Th=\sum_n T^nh$.\par

We want to prove that $(\xymatrix{\bar{T}(M)\ar@<.4ex>[r]^{T(\imath)}&\bar{T}(N)
\ar@<.4ex>[l]^{T(\pi)}}, Th)$ is a coalgebra contraction, i.e. that
\[ (T(\imath\pi)\otimes Th+Th\otimes\operatorname{Id})\circ\copa=
\copa\circ Th.\]
Let $n$ be a fixed positive integer, writing 
\[ T^nh=\sum_{i=1}^{n}T^n_ih\;,\qquad
T^n_ih=(\imath\pi)^{\otimes i-1}\otimes h\otimes \Id_{N}^{\otimes n-i},\]
for every $i=1,\ldots,n$ we have
\[ \copa \circ T^n_ih=\sum_{j=1}^{i-1}(\imath\pi)^{\otimes j}\otimes T^{n-j}_{i-j}h+
\sum_{j=i}^{n-1} T^{j}_{i}h\otimes \Id_N^{\otimes n-j}.\]
Therefore
\begin{align*}\copa \circ T^nh=
\sum_{i=1}^n\copa \circ T^n_ih
&=\sum_{i=1}^n\sum_{j=1}^{i-1}(\imath\pi)^{\otimes j}\otimes T^{n-j}_{i-j}h+
\sum_{i=1}^n\sum_{j=i}^{n-1} T^{j}_{i}h\otimes \Id_N^{\otimes n-j}\\
&=\sum_{j=1}^{n-1}\sum_{i=j+1}^{n}(\imath\pi)^{\otimes j}\otimes T^{n-j}_{i-j}h+
\sum_{j=1}^{n-1}\sum_{i=1}^{j} T^{j}_{i}h\otimes \Id_N^{\otimes n-j}\\
&=\sum_{j=1}^{n-1}(\imath\pi)^{\otimes j}\otimes (\sum_{i=1}^{n-j} T^{n-j}_{i}h)+
\sum_{j=1}^{n-1}(\sum_{i=1}^{j} T^{j}_{i}h)\otimes \Id_N^{\otimes n-j}\\
&=\sum_{j=1}^{n-1}(\imath\pi)^{\otimes j}\otimes T^{n-j}h+
\sum_{j=1}^{n-1}T^{j}h\otimes \Id_N^{\otimes n-j}.
\end{align*}

It is now sufficient to sum over $n$.\end{example}

\bigskip
\section{Review of ordinary homological perturbation theory}

\textbf{Convention:} \emph{In order to simplify the notation, from now on, and unless otherwise stated, for every 
contraction $(\xymatrix{M\ar@<.4ex>[r]^\imath&N\ar@<.4ex>[l]^\pi}, h)$ we assume that 
$M$ is a submodule of $N$ and $\imath$ the inclusion.} 

\bigskip\par

Given a contraction $(\xymatrix{M\ar@<.4ex>[r]^\imath&N\ar@<.4ex>[l]^\pi}, h)$ of differential graded $R$-modules and a morphism $\de\in \Hom^1_R(N,N)$, the ordinary homological perturbation theory consists is a series of statements about the maps
\begin{equation}\label{equ.inclusioneperturbata}
\imath_{\de}=\sum_{n\ge 0}(h\de)^n\imath\;\in \Hom^0_R(M,N),
\end{equation}

\begin{equation}
\qquad\pi_{\de}=\sum_{n\ge 0}\pi(\de h)^n\;\in \Hom^0_R(N,M),
\end{equation}

\begin{equation}
D_{\de}=\pi\de\imath_{\de}=\pi_{\de}\de\imath\;\in \Hom^1_R(M,M),
\end{equation}

%\begin{equation}
%h_{\de}=
%\sum_{n\ge 0} (h\de)^n h=\sum_{n\ge
%0}h(\de h)^n\;\in \Hom^{-1}_R(N,N).
%\end{equation}

In order to have the above maps defined we need to impose some extra assumption. This may done by considering filtered contractions of complete modules (as in \cite{HK}) or 
by imposing a sort of local nilpotency for the operators $h\de, \de h$.

\begin{definition}\label{def.mezzoperturbation}
Given a contraction  $(\xymatrix{M\ar@<.4ex>[r]^\imath&N\ar@<.4ex>[l]^\pi},h)$  denote
\[ \mathcal{N}(N,h)=\{\de\in \Hom^1_R(N,N)\mid \cup_n\ker((h\de)^n\imath)
=M,\; \cup_n\ker(\pi(\de h)^n)=N\}.\]
\end{definition}

It is plain that the  maps $\imath_{\de}, \pi_{\de}$ and $D_{\de}$ are well defined for every $\de\in\mathcal{N}(N,h)$. Moreover they are functorial in the following sense: given a morphism of contractions
\[ f\colon  (\xymatrix{M\ar@<.4ex>[r]^\imath&N\ar@<.4ex>[l]^\pi},h)
\to (\xymatrix{A\ar@<.4ex>[r]^i&B\ar@<.4ex>[l]^p},k)\]
and two elements $\de\in \mathcal{N}(N,h)$, $\delta\in \mathcal{N}(B,k)$ such that
$f\de=\delta f$ we have 
\[f\imath_{\de}=\sum_{n\ge 0}f(h\de)^n\imath=\sum_{n\ge 0}(k\delta)^nf\imath=
\sum_{n\ge 0}(k\delta)^ni\hat{f}=i_{\delta}\hat{f}.\]
Similarly we have 
$\hat{f}\pi_{\de}=p_{\delta}f$, 
$\hat{f}D_{\de}=D_{\delta}\hat{f}$.

\begin{lemma}\label{lem.injectivity}  Let $(\xymatrix{M\ar@<.4ex>[r]^\imath&N\ar@<.4ex>[l]^\pi},h)$ be a contraction and $\de\in \mathcal{N}(N,h)$. Then $\imath_{\de}$ is injective and
\[ \pi_{\de}\imath_{\de}=\pi\imath=\operatorname{Id}_M.\]
 \end{lemma}
 
\begin{proof} Immediate consequence of annihilation properties.
It is useful to point out that the proof of the injectivity of $\imath_{\de}$ does not depend on the annihilation properties. 
Assume $\imath_{\de}(x)=0$ and let $s\ge 0$ be the  
minimum integer such that $(h\de)^s\imath(x)=0$. If $s>0$ then 
\[ 0=(h\de)^{s-1}\imath_{\de}(x)=(h\de)^{s-1}\imath(x)+
\sum_{k\ge s}(h\de)^{k}\imath(x)= (h\de)^{s-1}\imath(x)\]
giving a contradiction. Hence $s=0$ and $\imath(x)=0$. 
\end{proof}

\begin{proposition}\label{prop.compositioncompatibility} 
The formula \ref{equ.inclusioneperturbata}  is compatible with composition of contractions. More precisely, if 
\[ (\xymatrix{L\ar@<.4ex>[r]^i&M\ar@<.4ex>[l]^p}, k)\circ
(\xymatrix{M\ar@<.4ex>[r]^\imath&N\ar@<.4ex>[l]^\pi}, h)
=
(\xymatrix{L\ar@<.4ex>[r]^{\imath i}&N\ar@<.4ex>[l]^{p\pi }}, h+\imath k \pi)
\]
then $(\imath i)_{\de}=\imath_{\de} i_{D_{\de}}$, 
provided that all terms of the equation are defined.
\end{proposition}

\begin{proof} We have
\begin{align*}
\imath_{\de} i_{D_{\de}}=&\sum_{n\ge 0}(h\de)^n\imath 
\sum_{m\ge 0}(kD_{\de})^m i\\
&=\sum_{n\ge 0}(h\de)^n 
\sum_{m\ge 0}\imath(k\pi\de \sum_{s\ge 0}(h\de)^s\imath)^m i\\
&=\sum_{n\ge 0}(h\de)^n 
\sum_{m\ge 0}(\imath k\pi\de \sum_{s\ge 0}(h\de)^s)^m\imath i\\
&=\sum_{n\ge 0}(h\de+\imath k\pi\de )^n\imath i\\
&=(\imath i)_{\de}.
\end{align*}
\end{proof}

\begin{proposition}\label{prop.huebKad}
Let $(\xymatrix{M\ar@<.4ex>[r]^\imath&N\ar@<.4ex>[l]^\pi},h)$ be a coalgebra contraction
and  $\de\in \mathcal{N}(N,h)$. If $\de$ is a coderivation then 
$\imath_{\de}$ and $\pi_{\de}$ are morphisms of graded coalgebras and $D_{\de}$ is a coderivation.
\end{proposition}

\begin{proof} Consider the contraction
\[(\xymatrix{M\otimes_RM\ar@<.4ex>[r]^{\imath\otimes\imath}&
N\otimes_RN\ar@<.4ex>[l]^{\pi\otimes\pi}},k), \quad\text{ where }\quad 
k=h*h=\imath\pi\otimes h+h\otimes \Id_N,\]
and $\delta=\de\otimes\Id_N+\Id_N\otimes\de$. In order to prove  that  $\delta\in \mathcal{N}(N\otimes_RN,k)$ we  show that for every integer $n\ge 0$ we have
\[ (k\delta)^n(\imath\otimes\imath)=
\sum_{i+j=n}(h\de)^i\imath\otimes (h\de )^j\imath\;,\qquad
 (\pi\otimes\pi)(\delta k)^n=
\sum_{i+j=n}\pi(\de h)^i\otimes\pi (\de h)^j\;.
\]
We prove here only the first equality by induction on $n$; the second  is completely similar and left to the reader. Since
\[ k\delta=h\de\otimes\Id_N+h\otimes\de-\imath\pi\de\otimes
h+\imath\pi\otimes h\de,\]%
according to annihilation properties we have:
\[h\de\otimes\Id_N\left(\sum_{i+j=n}(h\de )^i\imath\otimes(h\de
)^j\imath\right)=\sum_{i+j=n}(h\de )^{i+1}\imath\otimes(h\de
)^j\imath,\]
\[h\otimes\de\left(\sum_{i+j=n}(h\de )^i\imath\otimes(h\de )^j\imath\right)
=0,\qquad
\imath\pi\de\otimes h\left(\sum_{i+j=n}(h\de )^i\imath\otimes(h\de
)^j\imath\right)=0,\]
\[\imath\pi\otimes h\de\left(\sum_{i+j=n}(h\de )^i\imath\otimes(h\de
)^j\imath\right)=\imath\otimes (h\de )^{n+1}\imath.\]
Therefore
\[ (\imath\otimes\imath)_{\delta}=
\sum_{n\ge 0}(k\delta)^n(\imath\otimes\imath)=
\sum_{i,j\ge 0}(h\de)^i\imath\otimes (h\de )^j\imath=
\imath_{\de}\otimes\imath_{\de}\;,
\]
\[ (\pi\otimes\pi)_{\delta}=
\sum_{n\ge 0}(\pi\otimes\pi)(\delta k)^n=
\sum_{i,j\ge 0}\pi(\de h)^i\otimes \pi(\de h)^j=
\pi_{\de}\otimes\pi_{\de}\;.
\]

Denoting by $\Delta\colon N\to N\otimes_R N$ the coproduct, 
since $\de$ is a coderivation we have $\delta\Delta=\Delta\de$; since $\Delta$ is a morphism of contractions we have by functoriality
\[ \Delta\imath_{\de}=(\imath\otimes\imath)_{\delta}\hat{\Delta}=
(\imath_{\de}\otimes\imath_{\de})\hat{\Delta},\qquad \hat{\Delta}\pi_{\de}=(\pi\otimes\pi)_{\delta}\Delta=
(\pi_{\de}\otimes\pi_{\de})\Delta,\]
and then $\imath_{\de}, \pi_{\de}$ are morphisms of coalgebras.
Finally $D_{\de}$ is a coderivation because it is the composition of the coderivation $\de$ and the two morphisms of coalgebras $\imath_{\de}$ and $\pi$.
\end{proof}

A proof of Proposition~\ref{prop.huebKad} is given in \cite{HK} under the unnecessary assumption that 
$(d+\de)^2=0$.

\begin{definition} Let $N$ be a differential graded $R$-module. 
A \emph{perturbation} of the differential $d_N$ 
is a linear map $\partial\in \Hom^1_R(N,N)$ such that
$(d_N+\partial)^2=0$.
\end{definition}

\begin{theorem}[Ordinary perturbation lemma]\label{thm.pertlemma}
Let
$(\xymatrix{M\ar@<.4ex>[r]^\imath&N\ar@<.4ex>[l]^\pi},h)$ be a contraction  and let $\de\in
\mathcal{N}(N,h)$ be a perturbation  of the differential $d_N$.
Then $D_{\de}$ is a perturbation of
$d_{M}=\pi d_N\imath$ and
\[ \pi_{\de}\colon (N,d_N+\de)\to (M,d_M+D_{\de}),\qquad
\imath_{\de}\colon (M,d_M+D_{\de})\to (N,d_N+\de)\]%
are morphisms of differential graded $R$-modules.
\end{theorem}

\begin{proof} See \cite{HK,jimmurra} and references therein for proofs and history. We prove again this result in Remark~\ref{rem.cpl} as a particular case of  the relative perturbation lemma.
\end{proof}

\begin{remark} If  $\cup_n\ker(h\de)^n=N$, and $\de$ is a perturbation of $d_N$,
then $\imath_{\de}$ is the unique morphism of  graded $R$-modules 
$M\to N$ whose image is a subcomplex of $(N,d_N+\de)$ and satisfying the ``gauge fixing'' condition
\[ h\imath_{\de}=0,\qquad \pi\imath_{\de}=\operatorname{Id}_M.\]
In fact $h(d_N+\de)\imath_{\de}=0$
and then
\begin{align*}
\imath_{\de}=&\imath_{\de}+hd_N\imath_{\de}+h\de\imath_{\de}=(\imath\pi-d_Nh)\imath_{\de}+
h\de\imath_{\de}=\imath+(h\de)\imath_{\de}\\
=&(\operatorname{Id}_N-h\de)^{-1}\imath.
\end{align*}
Similarly $\pi_{\de}$ is the unique morphism of  graded $R$-modules 
$M\to N$ whose kernel is a subcomplex of $(N,d_N+\de)$ and satisfying
\[ \pi_{\de}h=0,\qquad \pi_{\de}\imath=\operatorname{Id}_M.\]
\end{remark}

The coalgebra perturbation lemma cited in the abstract is obtained by putting together 
Proposition~\ref{prop.huebKad} and Theorem~\ref{thm.pertlemma}.

\medskip
\section{The relative  perturbation lemma}
\label{sec.relativepertlemma}

\begin{definition}\label{def.relativeperturbation}
Let $N$ be a differential graded $R$-module and $A\subset N$ a
differential graded  submodule. A morphism $\de\in \Hom^1_R(N,N)$ is called a \emph{perturbation of $d_N$ over $A$} if
\[ \de
(A)\subset A\quad\text{and}\quad (d_N+\de)^2(A)=0.\]
\end{definition}

\begin{remark} The meaning of Definition~\ref{def.relativeperturbation} becomes more clear when we impose some extra assumption on $\de$. For instance, if $N$ is a differential graded coalgebra and $\de$ is a coderivation, then in general does not exist any coderivation $\delta$ of $N$ such that $\delta_{|A}=\de_{|A}$ and 
$(d_N+\delta)^2=0$. An explicit example of this phenomenon will be described in 
Section~\ref{sec.coderivation}.\end{remark}

\begin{theorem}[Relative perturbation lemma]\label{thm.relapert}
Let $(\xymatrix{M\ar@<.4ex>[r]^\imath&N\ar@<.4ex>[l]^\pi},h)$ be a contraction with $M\subset N$ and $\imath$ the inclusion. Let $A\subset N$ be a differential graded submodule 
and $\de\in \mathcal{N}(N,h)$  a perturbation of $d_N$
over $A$. Assume moreover that:
\begin{enumerate}

\item $\pi(A)\subset A\cap M$.

\item $\imath_{\de}(A\cap M)\subset A$.
\end{enumerate}
Then
\[ D_{\de}=\sum_{n\ge 0}\pi\de (h\de)^n\imath=\sum_{n\ge 0}\pi(\de
h)^n\de\imath\in \Hom^1_R(M,M),\]
is a perturbation of $d_M$ over $A\cap M$ and
\[ \imath_{\de}=\sum_{n\ge 0}(h\de)^n\imath\colon (A\cap M,d_M+D_{\de})\to (A,d_N+\de)\]%
is a morphisms of differential graded $R$-modules.
\end{theorem}

\begin{remark} It is important to point out that 
we do not require that $h(A)\subset A$ but only the weaker assumption 
$\imath_{\de}(M\cap A)\subset A$.\end{remark}

\begin{proof} We first note that 
$D_{\de}=\pi\de\imath_{\de}$ and then 
$D_{\de}(A\cap M)\subset A\cap M$. In order to simplify the notation we denote $d=d_N$ and $I=Id_N$.
 Setting $\psi=\de^2+d\de+\de d\in \Hom^2_R(N,N)$ we have the formula
\begin{equation}\label{equ.l4}
 \sum_{n,m\ge 0}(\de h)^n\de\imath\pi\de(h\de)^m=
\sum_{n,m\ge 0}(\de h)^n\psi(h\de)^m
-\sum_{m\ge 0} d\de(h\de)^m-\sum_{n\ge 0}(\de h)^n\de d.
\end{equation}
In fact, since $\imath\pi=I+hd+dh$, we have
\[ \de\imath\pi\de=\de(I+hd+dh)\de=\de^2+\de hd\de+\de dh\de=
\psi-(I-\de h)d\de-\de d(I-h\de)\]
and therefore
\begin{align*}
\sum_{n,m\ge 0}&(\de h)^n\de\imath\pi\de(h\de)^m\\
=&\sum_{n,m\ge 0}(\de
h)^n\psi(h\de)^m-\sum_{n,m\ge 0}(\de h)^n(I-\de h)d\de(h\de)^m
 -\sum_{n,m\ge 0}(\de h)^n\de d(I-h\de)(h\de)^m\\
 =&\sum_{n,m\ge 0}(\de
h)^n\psi(h\de)^m-\sum_{m\ge 0}d\de(h\de)^m
 -\sum_{n\ge 0}(\de h)^n\de d\;.
\end{align*}
We have
\begin{align*}
(d+\de)\imath_{\de}=&\sum_{m\ge
0}d(h\de)^m\imath+\sum_{m\ge 0}\de(h\de)^m\imath
=d\imath+\sum_{m\ge
0}dh\de(h\de)^m\imath+\sum_{m\ge 0}\de(h\de)^m\imath\\
=&d\imath +\sum_{m\ge 0}(I+dh)\de(h\de)^m\imath
=d\imath+\sum_{m\ge 0}(\imath\pi-hd)\de(h\de)^m\imath\;,
\end{align*}

\begin{align*}
\imath_{\de}(d_M+D_{\de})\!=&
\sum_{n\ge 0}(h\de)^n\imath d_M+\sum_{n,m\ge 0}(h\de)^n\imath\pi\de(h\de)^m\imath\\
=&\sum_{n\ge 0}(h\de)^n\imath d_M+\!
\sum_{m\ge 0}\imath\pi\de(h\de)^m\imath+
h\sum_{n,m\ge 0}(\de h)^n\de\imath\pi\de(h\de)^m\imath\qquad\qquad\\
=&\sum_{n\ge 0}(h\de)^nd\imath+\!\sum_{m\ge 0}\imath\pi\de(h\de)^m\imath
+\!\sum_{n\ge 0}h(\de h)^n\psi\imath_{\de}
-\!\sum_{m\ge 0} hd\de(h\de)^m\imath-\!\sum_{n\ge 0}h(\de h)^n\de d\imath\\
=&d\imath+\sum_{m\ge 0}(\imath\pi-hd)\de(h\de)^m\imath
+\sum_{n\ge 0}h(\de h)^n\psi\imath_{\de},
\end{align*}
and therefore
\[ \imath_{\de}(d_M+D_{\de})-(d+\de)\imath_{\de}=\sum_{n\ge 0}h(\de h)^n\psi\imath_{\de}.\]
In particular,
for every $x\in M\cap A$ we have $\psi\imath_{\de}(x)=0$ and then
\[ \imath_{\de}(d_M+D_{\de})(x)=(d+\de)\imath_{\de}(x).\]
Now we prove that $D_{\de}$ is perturbation of $d_M$ over $M\cap A$, i.e. that
$(d_M+D_{\de})^2x=0$
for every $x\in M\cap A$. Since $\pi h=0$ we have $\pi\imath_{\de}=\pi\imath$ and then
$\imath_{\de}$ is injective.
If $x\in M\cap A$ we have
\[  \imath_{\de}(d_M+D_{\de})^2x=(d+\de)\imath_{\de}(d_M+D_{\de})x=
(d+\de)^2\imath_{\de}x=0.\]
\end{proof}

\begin{remark}\label{rem.cpl} In the set-up of Theorem~\ref{thm.relapert},  if $h(A)\subset A$ then also $\pi_{\de}\colon (A,d+\de)\to (A\cap M, d_M+D_{\de})$ is a morphism of differential graded $R$-modules. In fact, under this additional assumption we have
\[ \pi_{\de}(A)=\sum_{n\ge0} \pi(\de h)^n(A)\subset A\cap M,\qquad 
\sum_{n,m\ge 0}(\de h)^n\psi(h\de)^m h(A)=0,\]
and therefore in $A$ the following equalities hold: 
 
\begin{multline*}
\pi_{\de}(d+\de)=\sum_{n\ge 0}\pi (\de h)^n d+\sum_{n\ge 0}\pi
(\de h)^n\de=\pi d+\sum_{n\ge 0}\pi (\de h)^n\de hd+\sum_{n\ge
0}\pi (\de h)^n\de\\
=\pi d+\sum_{n\ge 0}\pi (\de h)^n\de(I+hd)= \pi d+\sum_{n\ge 0}\pi
(\de h)^n\de(\imath\pi-dh).
\end{multline*}

\begin{align*}
(d+D_{\de})\pi_{\de}=&\sum_{n\ge 0}\pi d(\de h)^n+ \sum_{n,m\ge 0}\pi(\de
h)^n\de\imath\pi(\de h)^m\\
=&\sum_{n\ge 0}\pi d(\de h)^n+ \sum_{n\ge 0}\pi(\de
h)^n\de\imath\pi +\sum_{n\ge 0,m\ge 1}\pi(\de
h)^n\de\imath\pi(\de h)^m\\
=&\sum_{n\ge 0}\pi d(\de h)^n+ \sum_{n\ge 0}\pi(\de
h)^n\de\imath\pi +\sum_{n,m\ge 0}\pi(\de
h)^n\de\imath\pi\de (h\de)^mh\\
=&\sum_{n\ge 0}\pi d(\de h)^n+ \sum_{n\ge 0}\pi(\de
h)^n\de\imath\pi - \sum_{m\ge 0}\pi d\de(h\de)^mh
 -\sum_{n\ge 0}\pi (\de h)^n\de dh\\
=&\left(\sum_{n\ge 0}\pi d(\de h)^n- \sum_{m\ge 0}\pi
d\de(h\de)^mh\right)+\sum_{n\ge 0}\pi(\de h)^n\de\imath\pi
 -\sum_{n\ge 0}\pi (\de h)^n\de dh\\
=& \pi d+\sum_{n\ge 0}\pi (\de h)^n\de(\imath\pi-dh).
\end{align*}
\end{remark}

\begin{remark} It is straightforward to verify that all the previous proofs also work for 
the weaker notion of contraction where the condition $\pi\imath=Id_M$ is replaced with 
\emph{$\imath$ is injective and $\imath(M)$ is a direct summand of $N$ as graded $R$-module.}
\end{remark}

\bigskip
\section{Review of reduced symmetric coalgebras and their coderivations}
\label{sec.coderivation}

From now on we assume that $R=\K$ is a field of characteristic 0. Given a graded vector space
$V$, the \emph{twist map}
\[\twist\colon V\otimes V\to V\otimes V,\qquad 
\twist(v\otimes w)=(-1)^{\deg(v)\deg(w)}w\otimes v,\] 
extends naturally to an action of the symmetric group $\Sigma_n$ on the tensor product
$\bigotimes^nV$:
\[ \sigma_{\twist}(v_{1}\otimes\cdots\otimes v_{n})=
\pm
\;v_{\sigma^{-1}(1)}\otimes\cdots\otimes v_{\sigma^{-1}(n)},\qquad \sigma\in \Sigma_n.\]
We will denote by $\bigodot^nV=(\bigotimes^nV)^{\Sigma_n}$ the subspace of invariant tensors. Notice that if $W\subset V$ is a graded subspace, then $\bigodot^nW=\bigodot^nV\cap \bigotimes^nW$.
It is  easy to see that the subspace
\[ \bar{S}(V)=\somdir_{n=1}^{\infty}\symmetric^nV\subset 
\somdir_{n=1}^{\infty}\tensor^nV=\bar{T}(V)\]
is a graded subcoalgebra, called the \emph{reduced symmetric coalgebra} generated by $V$.
Let's denote by  $p\colon \bar{T}(V)\to V$ the projection; we will also denote by 
$p\colon \bar{S}(V)\to V$ the restriction of the projection to symmetric tensors. The following well known properties hold (for proofs see e.g. \cite{defomanifolds}):
\begin{enumerate}

\item Given a morphism of graded coalgebras $F\colon \bar{T}(V)\to \bar{T}(W)$ we have
$F(\bar{S}(V))\subset \bar{S}(W)$.

\item  Given a morphism of graded vector spaces $f\colon \bar{T}(V)\to W$
there exists an unique morphism of graded coalgebras $F\colon \bar{T}(V)\to \bar{T}(W)$ such that $f=pF$.

\item  Given a morphism of graded vector spaces $f\colon \bar{S}(V)\to W$
there exists an unique morphism of graded coalgebras $F\colon \bar{S}(V)\to \bar{S}(W)$ such that $f=pF$.
\end{enumerate}

Similar results hold for coderivations. More precisely for
every map $q\in\Hom^k(\bar{T}(V),V)$ there exists an unique coderivation
$Q\colon \bar{T}(V)\to \bar{T}(V)$ of degree $k$ such that $q=pQ$. The coderivation $Q$ is given by the explicit formula
\begin{equation}\label{equ.coder}
Q(a_{1}\otimes\cdots\otimes a_{n})
=\sum_{l=1}^{n}\sum_{i=0}^{n-l}(-1)^{k(\bar{a_{1}}+\cdots+\bar{a_{i}})}
a_{1}\otimes\cdots\otimes a_{i}\otimes
q(a_{i+1}\otimes\cdots\otimes a_{i+l})\otimes\cdots\otimes a_{n}\;,
\end{equation}
where $\bar{a_i}=\deg(a_i)$. Moreover $Q(\bar{S}(V))\subset \bar{S}(V)$ and the restriction
of $Q$ to $\bar{S}(V)$ depends only on the restriction of $q$ on $\bar{S}(V)$. In particular every coderivation of $\bar{S}(V)$ extends to a coderivation of $\bar{T}(V)$.

\begin{definition} A coderivation $Q$ of degree $+1$ is called a \emph{codifferential} if 
$Q^2=0$.\end{definition}

\begin{lemma} A coderivation $Q$ of degree $+1$ is a codifferential if and only if
$pQ^2=0$.\end{lemma}

\begin{proof} The space of coderivations of a graded coalgebra is closed under the bracket
\[ [Q,R]=QR-(-1)^{\deg(Q)\deg(R)}RQ\]
and therefore if $Q$ is a coderivation of odd degree, then its square $Q^2=[Q,Q]/2$ is again a coderivation.\end{proof}

Every codifferential on $\bar{T}(V)$ induces by restriction a codifferential on
$\bar{S}(V)$. Conversely it is generally false that a codifferential on 
$\bar{S}(V)$ extends to a codifferential on $\bar{T}(V)$. This is well known to experts; however we will give here
an example of this phenomenon for the lack of suitable references.\par

We restrict our attention to graded vector spaces concentrated in degree $-1$, more precisely we assume that
$V=L[1]$, where $L$ is a vector space and $[1]$ denotes the shifting of the degree, i.e. 
$L[1]^{i}=L^{i+1}$.
Under this assumption every codifferential in $\bar{T}(V)$ (resp.: 
$\bar{S}(V)$) is determined by a linear map $q\colon \bigotimes^2 V\to V$ (resp.: 
$q\colon \bigodot^2 V\to V$) of degree $+1$.

\begin{lemma} In the above assumption:
\begin{enumerate}

\item The map  
\[L\times L\to L,\qquad xy=q(x\otimes y),\]
is an associative product if and only if $q$ induces a codifferential in $\bar{T}(V)$.

\item The map 
\[L\times L\to L,\qquad [x,y]=q(x\otimes y-y\otimes x)=xy-yx,\]
is a Lie bracket if and only if $q$ induces a codifferential in $\bar{S}(V)$.
\end{enumerate}
\end{lemma}

\begin{proof} We have seen that $Q$ is a codifferential in $\bar{T}(V)$  if and only if
$pQ^2=qQ\colon \tensor^3V\to V$ is the trivial map. It is sufficient to observe that  
\[qQ(x\otimes y\otimes z)=q(q(x\otimes y)\otimes z)-q(x\otimes q(y\otimes z))=
(xy)z-x(yz).\]
Similarly $Q$ is a codifferential in $\bar{S}(V)$  if and only if for every $x_1,x_2,x_3$ we have
\begin{align*}
0=&qQ\left(\sum_{\sigma\in \Sigma_3}(-1)^{\sigma}x_{\sigma(1)}\otimes x_{\sigma(2)}\otimes x_{\sigma(3)}\right)\\
=& \sum_{\sigma\in \Sigma_3}(-1)^{\sigma} ((x_{\sigma(1)}x_{\sigma(2)})x_{\sigma(3)}-
x_{\sigma(1)}(x_{\sigma(2)}x_{\sigma(3)}))\\
=&[[x_1,x_2],x_3]+[[x_2,x_3],x_1]+[[x_3,x_1],x_2]
\end{align*}
\end{proof}

Therefore every Lie bracket on $L$ not induced by an associative product gives a codifferential  on $\bar{S}(L[1])$ which does not extend to a codifferential on 
$\bar{T}(L[1])$.

\begin{example} Let $\K$ be a field of characteristic $\not=2$ and $L$ a vector space of dimension $3$ over $\K$ with basis $A,B,H$.
Then does not exist any associative product on $L$ such that
\[ AB-BA=H,\qquad HA-AH=2A,\qquad HB-BH=-2B.\]

We prove this fact by contradiction: assume that there exists an associative product as above, 
then the pair $(L,[,])$, where $[X,Y]=XY-YX$, is a Lie algebra isomorphic to $sl_2(\K)$. Writing 
\[ H^2=\gamma_1 A+\gamma_2 B+\gamma H\]
 we have
 \[ 0=[H^2,H]=\gamma_1 [A,H]+\gamma_2 [B,H]\]
and therefore $\gamma_1=\gamma_2=0$, $H^2=\gamma H$.  
Possibly acting with the Lie automorphism
\[ A\mapsto B,\qquad B\mapsto A,\qquad H\mapsto -H,\]
it is not restrictive to assume $\gamma\neq -1$.  

Since $[AH,H]=[A,H]H=-2AH$, writing $AH=xA+yB+zH$ for some $x,y,z\in\K$ we have
\[ 0=[AH,H]+2AH=x[A,H]+y[B,H]+2xA+2yB+2zH=4yB+2zH\]
giving $y=z=0$ and $AH=xA$. Moreover $2A^2=A[H,A]=[AH,A]=[xA,A]=0$ and then $A^2=0$.
Since 
\[ 0=A(H^2)-(AH)H=\gamma AH-x AH=(\gamma x-x^2)A\]
we have either $x=0$ or $x=\gamma$. In both cases $x\not=-1$ and  then $AH+HA=(2x+2)A\not=0$.
This gives a contradiction since
\[ -AH=A(AB-H)=ABA=(BA+H)A=HA.\]
\end{example}

\bigskip
\section{The $L_{\infty}$-algebra  perturbation lemma}

The bar construction gives an equivalence from the category of $L_{\infty}$-algebras 
and the category of differential graded reduced symmetric coalgebras (see e.g. \cite{cone,fuka,K}).\par

According to Formula~\ref{equ.coder}, every coderivation $Q\colon \bar{T}(V)\to \bar{T}(V)$ of degree $+1$ can be uniquely decomposed as
$Q=d+\de$, where
\[d(\tensor^n V)\subset \tensor^n V,\qquad
\de(\tensor^n V)\subset \somdir_{i=1}^{n-1}\tensor^i V,\qquad \forall\; n>0.\]
and
\[d(a_{1}\otimes\cdots\otimes a_{n})
=\sum_{i=0}^{n-1}(-1)^{\bar{a_{1}}+\cdots+\bar{a_{i}}}
a_{1}\otimes\cdots\otimes a_{i}\otimes
d_1(a_{i+1})\otimes a_{i+2}\otimes\cdots\otimes a_{n}\]
where $d_1=Q_{|V}\colon V\to V$.
If $Q$ is a codifferential on $\bar{T}(V)$ then $d^2(V)=0$, $d$ is the natural differential on  the tensor powers of the complex $(V,d_1)$ and $\de$ is a perturbation of $d$.

If $Q$ is a codifferential on $\bar{S}(V)$ then $d^2(V)=0$ and therefore $d$ is the natural differential on  the symmetric powers of the complex $(V,d_1)$ and $\de$ is a perturbation of $d$ over $\bar{S}(V)$.

\begin{theorem}\label{thm.symmcoalgebrapert}
In the above notation, let $Q=d+\de$ be a coderivation of degree +1 on 
$\bar{T}(V)$ which is a codifferential on $\bar{S}(V)$. Let $W$ be a differential
graded subspace of $(V,d)$ and let 
$(\xymatrix{W\ar@<.4ex>[r]&V\ar@<.4ex>[l]},k)$ be a contraction.
Taking the tensor power as in Example~\ref{exa.tensorcoalgebracontractions}, we get a coalgebra contraction 
$(\xymatrix{\bar{T}(W)\ar@<.4ex>[r]^\imath&\bar{T}(V)\ar@<.4ex>[l]^\pi},h)$ where
$h=Tk$. 
Setting 
\[ D_{\de}=\sum_{n\ge 0}\pi\de (h\de)^n\imath=\sum_{n\ge 0}\pi(\de
h)^n\de\imath\colon \bar{S}(W)\to \bar{S}(W),\]
then $d+D_{\de}$ is a codifferential in  $\bar{S}(W)$ and
\[ \imath_{\de}=\sum_{n\ge 0}(h\de)^n\imath\colon (\bar{S}(W),d+D_{\de})\to 
(\bar{S}(V),d+\de)\]%
is a morphisms of differential graded coalgebras.
\end{theorem}

\begin{proof}
Since $h(\bigotimes^n V)\subset \bigotimes^n V$ and $\de(\bigotimes^n V)\subset \bigoplus_{i=1}^{n-1}\bigotimes^i V$ we have
\[ \bigoplus_{i=1}^{n}\bigotimes^i V\subset \ker(\de h)^n\cap \ker(h\de)^n\]
and therefore $\de\in \mathcal{N}(\bar{T}(V),h)$.
According to Proposition~\ref{prop.huebKad} the maps 
\[ \imath_{\de}\colon \bar{T}(W)\to \bar{T}(V),\qquad
D_{\de}\colon \bar{T}(W)\to \bar{T}(W)\]
are respectively a morphism of graded coalgebras and a coderivation and then
\[ \imath_{\de}(\bar{S}(W))\subset\bar{S}(V),\qquad
D_{\de}(\bar{S}(W))\subset\bar{S}(W).\]
The conclusion now follows from Theorem~\ref{thm.relapert}, where
$N=\bar{T}(V)$, $M=\bar{T}(W)$ and $A=\bar{S}(V)$.
\end{proof}

\begin{remark} According to 
Proposition~\ref{prop.compositioncompatibility} the construction of Theorem~\ref{thm.symmcoalgebrapert} commutes with composition of contractions.
 \end{remark}

\begin{remark} In the notation of  Theorem~\ref{thm.symmcoalgebrapert}, if 
\[ S^nk\colon \symmetric^n V\to \symmetric^n V,\qquad S^nk=\frac{1}{n!}\sum_{\sigma\in\Sigma_n}\sigma_{\twist}\circ T^nk\circ
\sigma_{\twist}^{-1},\]
is the symmetrization of $T^nk$ and $Sk=\sum S^nk$, then 
$(\xymatrix{\bar{S}(W)\ar@<.4ex>[r]^\imath&\bar{S}(V)\ar@<.4ex>[l]^\pi},Sk)$
is a contraction but in general it is not a coalgebra contraction.
\end{remark}

In the set-up of Theorem~\ref{thm.symmcoalgebrapert} the map $\pi_{\de}\colon 
\bar{T}(V)\to \bar{T}(W)$ is a morphism of graded coalgebras and then induces a morphism
of graded coalgebras $\pi_{\de}\colon 
\bar{S}(V)\to \bar{S}(W)$ such that $\pi_{\de}\imath_{\de}$ is the identity on 
$\bar{S}(W)$. Unfortunately our proof does not imply that $\pi_{\de}$ is a morphism of 
complexes (unless $(d+\de)^2=0$ in $\bar{T}(V)$ or $D_{\de}=0$). However it follows from the homotopy classification of $L_{\infty}$-algebras \cite{K} that a morphism of differential graded coalgebras $\Pi\colon 
\bar{S}(V)\to \bar{S}(W)$ such that $\Pi\imath_{\de}=Id$ always exists.\par

We have proved that the map $\imath_{\de}\colon \bar{T}(W)\to \bar{T}(V)$ satisfies
the equation $\imath_{\de}=\imath+(h\de)\imath_{\de}$ and then 
$\imath_{\de}\colon \bar{S}(W)\to \bar{S}(V)$ is the unique morphism of symmetric graded coalgebras satisfying the recursive formula 
\begin{equation}\label{equ.recursive} \qquad\qquad
p\imath_{\de}=p\imath+kp\de\imath_{\de}\qquad\qquad (\text{where }p\colon \bar{S}(V)\to V \text{ is the  projection}). 
\end{equation}
It is possible  to prove that the validity of the Equation~\ref{equ.recursive} gives a combinatorial description of $\imath_{\de}$ as sum over rooted trees \cite{cone,fuka} and assures that 
$\imath_{\de}\colon (\bar{S}(W),d+\pi\de\imath_{\de})\to (\bar{S}(V),d+\de)$ is a morphism of differential graded coalgebras (see e.g. the arXiv version of \cite{cone}).

\end{document}